\documentclass{amsart}
\usepackage{amsmath, amssymb,epic,graphicx,mathrsfs,enumerate}
\usepackage[all]{xy}
\usepackage[margin=1.5cm]{geometry}

\usepackage{amsthm}
\usepackage{amssymb}
\usepackage{latexsym}
\usepackage{longtable}
\usepackage{epsfig}
\usepackage{amsmath}
\usepackage{hhline}
\usepackage{multicol}
\usepackage{xcolor}

 \DeclareMathOperator{\perm}{Sym}
 \DeclareMathOperator{\soc}{soc}
\DeclareMathOperator{\aut}{Aut} \DeclareMathOperator{\out}{Out}

 \DeclareMathOperator{\frat}{Frat}

\DeclareMathOperator{\alt}{Alt}

\newtheorem{thm}{Theorem}[section]
\newtheorem{cor}[thm]{Corollary}
 \newtheorem{lemma}[thm]{Lemma}
 
 \newtheorem{defn}[thm]{Definition}
 \newtheorem{con}[thm]{Conjecture}
\numberwithin{equation}{section}

\renewcommand{\footnote}{\endnote}
\newcommand{\ignore}[1]{}\makeglossary

\begin{document}

	\title[Independent generating sets]{Bounding the maximal size of independent generating sets of finite groups}

\author[A. Lucchini]{Andrea Lucchini}
\address{Andrea Lucchini, Dipartimento di Matematica  \lq\lq Tullio Levi-Civita\rq\rq,\newline
	University of Padova, Via Trieste 53, 35121 Padova, Italy} 
\email{lucchini@math.unipd.it}

\author[M. Moscatiello]{Mariapia Moscatiello}
\address{Mariapia Moscatiello, Dipartimento di Matematica  \lq\lq Tullio Levi-Civita\rq\rq,\newline
	University of Padova, Via Trieste 53, 35121 Padova, Italy} 
\email{mariapia.moscatiello@math.unipd.it}

\author[P. Spiga]{Pablo Spiga}
\address{Pablo Spiga, Dipartimento di Matematica Pura e Applicata,\newline
	University of Milano-Bicocca, Via Cozzi 55, 20126 Milano, Italy} 
\email{pablo.spiga@unimib.it}
 
\maketitle	
\begin{abstract}Denote by $m(G)$ the largest size of a minimal generating set of a finite group $G$. We estimate $m(G)$ in terms of $\sum_{p\in \pi(G)}d_p(G),$  where we are denoting by $d_p(G)$ the minimal number of generators of a Sylow $p$-subgroup of $G$ and by $\pi(G)$ the set of prime numbers dividing the order of $G$. 
      
\end{abstract}

\section{Introduction}
A generating set $X$ of a finite group $G$ is said to be minimal (or independent) if no proper subset of $X$ generates $G$. 
We denote by  $m(G)$ the largest size of a minimal generating set of $G$. 
First steps toward investigating $m(G)$ have been taken in the context of permutation groups. An exhaustive investigation has been done for finite symmetric groups~\cite{cc,w}, proving that $m(\text{Sym}(n)) = n - 1$ and giving a complete description of the  independent generating sets of $\text{Sym}(n)$ having cardinality $n-1$.
Partial results for some families of simple groups are in~\cite{sw}: it turns out that already in the case $G = \text{PSL}(2,q)$, the precise value of $m(G)$ is quite difficult to obtain. Further Apisa and Klopsch \cite{ak} proposed a natural \lq\lq classification problem'': given a non-negative integer $c$, characterize all finite groups $G$ such that $m(G)- d(G) \leq c$, where $d(G)$ is the minimal size of a generating set of $G$. In particular, they classified the finite groups for which the equality $m(G) = d(G)$ holds. During the same period the first author started in \cite{min1,min2} a systematic investigation of how $m(G)$ can be estimated for an arbitrary finite group $G$.

In 1989, Guralnick \cite{rg} and the first author \cite{al} independently proved
that, if all the Sylow subgroups of a finite group $G$ can be generated by $d$ elements, then $d(G)\leq d+1$. One may ask, if minded so, whether a similar result  holds also for $m(G).$ More precisely, denote by $d_p(G)$ the minimal number of generators of a Sylow $p$-subgroup of $G$. 
\begin{center}
{\em Is it possible to bound $m(G)$ as a  function of the numbers $d_p(G),$\\ with $p$ running through the prime divisors of the order of $G$?}
\end{center} As customary, we denote by $\pi(G)$ the set of prime divisors of the order of $G$. It can be easily seen that, if $G$ is a finite nilpotent group, then $m(G)=\sum_{p\in \pi(G)}d_p(G)$. For simplicity, we let $$\delta(G):=\sum_{p\in \pi(G)}d_p(G).$$ In a private communication to the first author, Keith Dennis has  conjectured that
$m(G)\leq \delta(G)$, for every finite group $G.$

This conjecture is true for soluble groups.
\begin{thm}\label{thrm:soluble}Let $G$ be a finite soluble group. Then $m(G)\le \delta(G)$.
\end{thm}
\begin{proof}
In \cite{min1}, it is proved that $m(G)= \sum_{p\in \pi(G)}\alpha_p(G)$, where $\alpha_p(G)$ denotes the number of complemented factors of $p$-power order in a chief series of $G$.  Now, an easy inductive argument on the order of $G$ shows that $\alpha_p(G)\leq d_p(G)$ (see for example \cite[Lemma 4]{expected}). Therefore $m(G)\leq \sum_{p\in \pi(G)}d_p(G)=\delta(G)$.
\end{proof}

Despite Theorem~\ref{thrm:soluble}, Dennis' conjecture is false if $G$ is a symmetric group.
We  study the asymptotic behaviour of the function $\delta(\perm(n))$ in Section \ref{conti}. We prove in Theorem~\ref{stop} that $\delta(\perm(n))=\log 2\cdot n+o(n)$. Since $m(\perm(n))=n-1$ by~\cite{w}, the difference $m(\perm(n))-\delta(\perm(n))$ goes to infinity with $n$ and the inequality 
$m(\perm(n))\leq \delta(\perm(n))$ is satisfies only by finitely many values of $n$. Indeed, using the explicit upper bound on $\delta(\perm(n))$ in Theorem~\ref{stop} and some calculations, we have 
$$\begin{aligned}\delta(\perm(n))&=n-1\text{ if and only if } n\in\{1, 2, 3, 4, 5, 8, 10, 11, 16, 17, 18, 19, 25, 30, 31\},\\
\delta(\perm(n))&=n\text{ if and only if } n\in\{6, 7, 12, 13, 20, 26, 42, 43, 48\},\\
\delta(\perm(n))&=n+1\text{ if and only if } n\in\{
14, 21, 44, 45\},\\
\delta(\perm(n))&=n+2\text{ if and only if } n\in\{	15, 22, 23, 24, 46, 47\}.
\end{aligned}$$
For all the other values of $n$,  we have	$\delta(\perm(n))<n-1=m(\perm(n)).$

\

The proof of Theorem~\ref{stop} is rather technical and uses some explicit bounds on the prime counting function. However,  in  Lemma \ref{stup} we show by elementary means that, for every positive real number $\eta > 1,$ there exists a constant $c_\eta$ such that $m(\perm(n))=n-1\leq   c_\eta(\delta(\perm(n))^\eta$, for every $n\in \mathbb N$.

This motivates the following conjecture, which can be seen as a natural generalization of Dennis' conjecture.
\begin{con}\label{con}There exist two constants $c$ and $\eta$ such that $m(G)\leq c\cdot \delta(G)^\eta$ for every finite group $G.$
\end{con}

Given a normal subgroup $N$ of a finite group $G$, we let $$m(G,N)=m(G)-m(G/N).$$
The main result of this paper is the following theorem.
\begin{thm}\label{main} Let $G$ be a finite group and assume that there exist two constants $\sigma\geq 1$ and $\eta\geq 2$ such that
	$m(X,S)\leq \sigma\cdot|\pi(S)|^\eta$, for every composition factor $S$ of $G$ and for  every almost simple group $X$ with $\soc X=S.$ 
	Then $m(G)\leq \sigma \cdot \delta(G)^{\eta}.$
\end{thm}  

Theorem~\ref{main} reduces Conjecture \ref{con} to the following conjecture on finite almost simple groups. 

\begin{con}\label{co2}
	There exist two constants $\sigma$ and $\eta$ such that
	$m(X,\soc X)\leq \sigma\cdot |\pi(\soc X)|^\eta,$ for every finite almost simple group $X.$
\end{con}

Conjecture~\ref{co2} holds true when $\soc X$ is an alternating group or a sporadic simple group. Therefore, we have the following corollary. 

\begin{cor}\label{coro}
	There exists a constant $\sigma$ such that, if $G$ has no composition factor of Lie type, then $m(G)\leq \sigma \delta(G)^2.$
\end{cor}

Very little is known about $m(G),$ when $G$ is an almost simple group with socle a simple group of Lie type. Whiston and Saxl proved that, if $G=\mathrm{PSL}(2,q)$ with $q=p^r$ and with $p$ a prime number, then $m(G)\leq \max(6,\tilde \pi(r)+2)$ where $\tilde \pi(r)$ is the number of distinct prime divisors of $r.$ It follows from Zsigmondy's Theorem that $\tilde \pi(r)\leq \tilde \pi(q+1)\leq |\pi(\mathrm{PSL}(2,q))|$. Therefore Conjecture~\ref{co2} holds true when $G=\mathrm{PSL}(2,q).$ In his PhD thesis~\cite{pjk}, P. J. Keen  found a good upper bound for $m(\mathrm{SL}(3,q))$, when $q=p^r$ and $p$ is odd. In preparation for this, he also investigated the sizes of independent sets in $\mathrm{SO}(3,q)$ and $\mathrm{SU}(3,q)$, getting in all the cases a linear  bound in terms of $\tilde \pi(r).$  These partial results lead to conjecture that, if $\soc(X)$ is a group of Lie type of rank $n$ over the field with $q=p^r$ elements, then $m(X,\soc X)$ is polynomially bounded in terms of $n$ and $\tilde \pi(r)$. If this were true, then Conjecture \ref{co2} would also be true.

\

The proofs of Theorem~\ref{main} and Corollary~\ref{coro} are in Section~\ref{proofs}. These proofs require two preliminary results, one concerning the prime divisors of the order of a finite non-abelian simple group and the other about permutation groups, proved respectively in Sections~\ref{psimple} and~\ref{aux}.

\section{A result on the order of a finite simple group}\label{psimple}

For later use we need to recall some definitions and some results concerning
Zsigmondy primes.
\begin{defn}{\rm Let $a$ and $n$ be positive integers. A prime number $p$ is called a
	{\em{primitive prime divisor}} of $a^n-1$ if $p$ divides $a^n-1$
	and $p$ does not divide $a^e-1$ for every integer $1 \leq e \leq n-1$.
	We denote an arbitrary primitive prime divisors of $a^n-1$ by $a_n.$}
\end{defn}

\begin{thm}[Zsigmondy's Theorem \cite{z}]
	Let $a$ and $n$ be integers greater than $1$. There exists a primitive
	prime divisor of  $a^n-1$ except  in one of the following cases:
	\begin{enumerate}
		\item $n=2,\, a=2^s-1$ (i.e. $a$ is a Mersenne prime), and $s
		\geq 2$. \item $n=6,\, a=2$.
	\end{enumerate}
\end{thm}

\begin{lemma}\label{small}\cite[Proposition 5.2.15]{KL} $a_n\equiv 1 \mod n.$
\end{lemma}

\begin{thm}\label{almeno2}
	Let $S$ be a simple group of Lie type. There exist two different primes dividing $|S|$ but not $|\mathrm{Out}(S)|.$
\end{thm}
\begin{proof}
	Let $S=L(q)$ be a simple group of Lie type defined over the field with $q$ elements, where $q=p^t$ and $p$ is a prime number. From Burnside's theorem, $|\pi(S)|\ge 3$.  From~\cite{tre}, if $|\pi(S)|=3$,  then $$S\in \{A_1(5), A_1(7), A_1(8), A_1(17), A_2(9), A_2(3), ^2A_2(3), ^2A_3(2)\},$$ and for these groups the theorem holds by a direct inspection. Therefore, for the rest of the proof we may suppose
\begin{equation}\label{eq:eq4}
|\pi(S)|\ge 4.
\end{equation}
In particular, the result immediately follows when $|\pi(\mathrm{Out}(S))|\le 2$ and hence we may suppose $|\pi(\mathrm{Out}(S))|\ge 3$.	
	
	 The order of $L(q)$ has
	the cyclotomic factorization in terms of $q:$ 
	$$|L(q)|=\frac{1}{d}q^h\prod_{m\in \Lambda}\Phi_m(q)^{r_m},$$
	where $\Phi_m(q)$ is the $m$-th cyclotomic polynomial and $\Lambda$, $d$, $h$ and $r_m$ are listed in Tables L.1, C.1 and C.2 of \cite{cyclo}.

Suppose that $S\ne D_4(q)$ and that $S$ is untwisted. From~\cite[page 207]{prime}, if $r\geq 2$ and $m\geq 1$ are integers such that  $r_{m\cdot t}$ is a primitive prime of $(r^t)^m-1$, then  $r_{m\cdot t}$ divides $\Phi_m(r^t)$. From this and from Zsigmondy's theorem, we  conclude that, except for the six cases listed below, there exist  $i,j \in \Lambda$ with $2\le i<j$ such that $x:=p_{i\cdot t}$ and $y:=p_{j\cdot t}$ are distinct primitive prime divisors. In particular, $x$ and $y$ are odd divisors of $|S|$ and are relatively prime to  $q-1$ because $i\geq 2$. 
	Moreover, by Lemma~\ref{small}, $x\equiv y \equiv 1\mod t$ and hence  $x$ and $y$ are relatively prime to $t.$ In particular, $x$ and $y$ are our required primes.
(The case $S=D_4(q)$ is special in this argument because $3$ is (potentially) an odd prime divisor of $|\mathrm{Out}(S)|$ not arising from field automorphisms.)

We are going to analyze  the groups for which the existence of $x$ and $y$  is not ensured from the previous argument.
	 \begin{enumerate}
		\item $S=A_2(q)$ and $q$ is a Mersenne prime: in this case $|\mathrm{Out}(S)|=2\cdot (q-1,3)$
		is divisible by at most 2 different primes, contradicting $|\pi(\mathrm{Out}(S))|\ge 3$.
		\item $S=A_2(4):$ in this case 5 and 7 are the required primes.
		\item $S=A_1(q):$ we may assume $t\geq 5$, otherwise $|\pi(\mathrm{Out} S)|\le 2$. Now, the existence of $x=p_{t}$ and $y=p_{2\cdot t}$ is ensured by 
		Zsigmondy's Theorem. 
		\item $S=B_2(q)$ with $q$ a Mersenne prime: in this case $|\pi(\mathrm{Out}( S))|=1,$ a contradiction.
		\item $S=B_2(8)$: in this case 5 and 7 are the requested primes.
		\item $S=G_2(q)$ with $q$ a Mersenne prime: in this case $|\pi(\mathrm{Out}(S))|\le 2,$ a contradiction.
	\end{enumerate}

It remains to deal with the case $S=D_4(q)$ and with the twisted groups of Lie type.

Suppose $S=D_4(q)$. Since $3$ divides $|\mathrm{Out}( S)|$, the previous argument fails exactly when the primitive prime divisor $x$ or $y$ is $3$. The existence of  $x=p_{2\cdot t}$, $y=p_{4\cdot t}$ and $z=p_{6\cdot t}$ is ensured when $q\notin\{2,8\}$ and when $q$ is not a Mersenne prime. When $q=2$, the result follows since $|\mathrm{Out}( S)|=6$; when $q=8$, we have that $t=3$ does not divide $y$ and $z$; therefore $y$ and $z$ are  prime numbers satisfying our statement . When $q$ is a Mersenne prime, if $q\neq 3,$ then  $q$ and $z$ are prime numbers satisfying our statement; if $q=3$, then $5$ and $7$ are prime numbers satisfying our statement.
	
	Assume $S\in \{{^2B}_2(q),  {^2G}_2(q), {^2F}_4(q)\}.$ In these cases we have $|\mathrm{Out}(S)|=t,$ so we may assume that $t$ is not a prime.
	Since the  existence of $x=p_{i\cdot t}$ and $y=p_{j\cdot t}$ is ensured by 
	Zsigmondy's Theorem, for two different elements $i$ and $j$ of $\Lambda$, we are done.
	
	If $S={{^3D}_4(q)}$ and $q\notin\{2,8\}$ and $q$ is not a Mersenne prime, then we can take $x=p_{2\cdot t}$ and $y=p_{6\cdot t}$ (notice that $|\out S|$ divides $3\cdot t$). When $q=2$ or  $q=8$ or  $q$ is a Mersenne prime, then $|\mathrm{Out}( S)|$ is divisible only by 3, against our assumption.

	If $S={^2E_6(q)}$, then we can take $x=p_{8\cdot t}$ and  $y=p_{12\cdot t}$ (notice that $|\mathrm{Out}( S)|$ divides $6\cdot t$).
	
	If $S={^2D_n(q)}$, then  $|\mathrm{Out}( S)|$ divides $8\cdot t$. So, when $q=2$ or when $q$ is a Mersenne prime, the result holds since $|\mathrm{Out} (S)|$ has only one prime divisor.  For the remaining cases, we  can take  $x=p_{4\cdot t}$ and $y=p_{6\cdot t}$.
		
	Finally assume $S={^2A}_n(q).$ In this case $|\mathrm{Out}( S)|=2\cdot t\cdot (n+1,q+1).$ If $n\geq 3$ and $q\neq 2,$ then we can take $x=p_{4\cdot t}$ and $y=p_{6\cdot t}.$ When $q=2$, we have $|\pi(\mathrm{Out}( {^2{A}}_n(2)))|\le 2$, which is a  contradiction.
	 We remain with the case $S={^2A}_2(q).$ The group $S={^2A}_2(3)$ was already analyzed, so we can suppose $q\geq 4$. Now $|\mathrm{Out}( S)|=2\cdot t \cdot (3,q+1)$, so we may assume $t\neq 1$.
	If $(3,q+1)=1,$  we may assume $t\neq 2$ and we can take  $x=p_{2\cdot t}$ and $y=p_{6\cdot t}$. Otherwise  $(3,q+1)=3,$ so $(3,q-1)=1$ and in particular $ x={p_t} \neq 3$. It follows that $x=p_{t}$ and $y=p_{6\cdot t}$ are the prime we are interested in. 
\end{proof}

\section{An auxiliary result}\label{aux}

\begin{lemma}\label{l:1}Let $Q$ be a $p$-group, let $P$ be a permutation $p$-group with domain $\Delta$ and let $n_\Delta(P)$ be the number of orbits of $P$ on $\Delta$. Then
	$$d(Q\mathrm{wr}_\Delta P)=d(P)+n_\Delta(P)d(Q).$$
\end{lemma}
\begin{proof}
	Let $\Delta_1,\ldots,\Delta_\ell$ be the orbits of $P$ on $\Delta$.
	
	Replacing $Q$ by $Q/\frat(Q)$ if necessary, we may suppose that $Q$ is an elementary abelian $p$-group. Let $B$ be the base group of the wreath product $W:=Q\mathrm{wr}_\Delta P$.
	
	Using the fact that $B$ is an abelian normal subgroup of $W$ and standard commutator computations,  we get 
	$[W,W]=[B,P][P,P].$ Given $\sigma \in P$ and $f\in B$, we have 
	\begin{align*}
	(\sigma f)^p&=\sigma^p f^{\sigma^{p-1}}f^{\sigma^{p-2}}\cdots f^\sigma f\\
	&=\sigma^p (f^{\sigma^{p-1}}f^{-1})(f^{\sigma^{p-2}}f^{-1})\cdots (f^\sigma f^{-1})\in P^p[B,P]
	\end{align*}
	and hence 
	\begin{equation}\label{eq:42meaning}\frat(W)=[B,P]\frat(P).
	\end{equation}
	
	Consider $V$, the subspace of $B$ consisting of all functions $g:\Delta\to Q$ with 
	$$\prod_{\delta\in \Delta_i}g(\delta)=1,\,\,\textrm{for every }i\in \{1,\ldots,\ell\}.$$
	Given $f\in B$, $\sigma\in P$ and $i\in \{1,\ldots,\ell\}$, we have
	\begin{align*}
	\prod_{\delta\in \Delta_i}[f,\sigma](\delta)&=\prod_{\delta\in \Delta_i}(f^\sigma f^{-1})(\delta)=\prod_{\delta\in \Delta_i}f(\delta^{\sigma^{-1}})f(\delta)^{-1}
	=\prod_{\delta\in \Delta_i}f(\delta^{\sigma^{-1}})\prod_{\delta\in \Delta_i}f(\delta)^{-1}\\
	&=\left(\prod_{\delta\in\Delta_i}f(\delta)\right)\left(\prod_{\delta\in \Delta_i}f(\delta)\right)^{-1}=1.
	\end{align*}
	Hence, $[B,P]\le V$. For each $i\in \{1,\ldots,\ell\}$, fix $\bar{\delta}_i\in \Delta_i$ and let $g\in V$. For every $i\in \{1,\ldots,\ell\}$ and $\delta\in \Delta_i\setminus\{\bar{\delta}_i\}$, we let $f_\delta:\Delta\to Q$ and $h_\delta:\Delta\to Q$ be the mappings defined by
	\begin{center}
		$f_\delta(\delta')$=	
		$\begin{cases}
		g(\delta) &\text{ if}\; \delta'=\delta, \\
		g(\delta)^{-1} &\text{ if} \; \delta'=\bar{\delta}_i, \\
		1 &\text{ if} \; \delta'\in \Delta\setminus\{\delta,\bar{\delta}_i\},
		\end{cases} \quad$
		$h_\delta(\delta')$=	
		$\begin{cases}
		g(\delta)^{-1} &\text{ if}\; \delta'=\delta, \\
		1 &\text{ if} \; \delta'\in \Delta\setminus\{\delta\}.
		\end{cases} $
	\end{center}
	Since $g\in V$, with a computation, we obtain
	$$g=\prod_{\delta\in \Delta\setminus\{\bar{\delta}_1,\ldots,\bar{\delta}_\ell\}}f_\delta.$$
	For each $i\in \{1,\ldots,\ell\}$ and $\delta\in \Delta_i\setminus\{\bar{\delta}_i\}$, since $\delta$ and $\bar{\delta}_i$ are in the same $P$-orbit, there exists $\sigma\in P$ with $\delta^\sigma=\bar{\delta}_i$. For each $\delta'\in \Delta$, we have
	\begin{align*}
	[h_\delta,\sigma](\delta')=h_\delta^{-1}(\delta')h_\delta^\sigma(\delta')=h_\delta(\delta')^{-1}h_\delta(\delta'^{\sigma^{-1}})=
	\begin{cases}
	g(\delta)&\textrm{if }\delta'=\delta,\\
	g(\delta)^{-1}&\textrm{if }\delta'=\bar{\delta}_i,\\
	1&\textrm{if }\delta'\in\Delta\setminus\{\delta,\bar{\delta}_i\}.
	\end{cases}
	\end{align*}
	It follows $f_\delta=[h_\delta,\sigma]\in [B,P]$ and hence $g\in [B,P]$. So, $V\le [B,P]$. Therefore \begin{equation}\label{eq:42answer}V=[B,P].
	\end{equation}
	
	From~\eqref{eq:42meaning},~\eqref{eq:42answer} and from the fact that $|B:V|=|Q|^\ell$, we obtain
	\begin{align*}
	|W:\frat(W)|&=|BP:V\frat(P)|=|P:\frat(P)||B:V|=p^{d(P)}|Q|^\ell\\
	&=p^{d(P)}p^{d(Q)\ell}=p^{d(P)+n_\Delta(P)d(Q)}.\qedhere
 	\end{align*}
\end{proof}

Given a permutation group $X$ on $\Omega$ and $\omega\in \Omega$, we let $X_\omega:=\{x\in X\mid \omega^x=\omega\}$ the stabilizer of $\omega$ in $X$. 
Let $K$ be a transitive permutation group on a set $\Omega$ and let $\omega\in \Omega$. We define $t_\Omega(K)$ to be the maximum number $t\in\mathbb{N}$ of subgroups $U_1,\dots, U_t$ of $K$ with
\begin{enumerate}
	\item\label{item1} $K_\omega=U_1\cap \cdots \cap U_t$,  and 
	\item\label{item2} $K_\omega\ne \bigcap_{j\in J}U_j$, for each proper subset $J$ of $\{1,\dots ,t\}$.
\end{enumerate} When~\eqref{item1} and~\eqref{item2} are satisfied (even if $t$ is not necessarily the maximum), we say that $U_1,\dots ,U_t$ are \emph{indipendent} subgroups of $K$. Moreover, let $S$ be a finite non-abelian simple group and let us denote by $\pi^*(S)$ the set of primes dividing $|S|$ but not $|\mathrm{Out}( S)|.$ 
\begin{thm}\label{pablo}Let $K$ be a transitive permutation on $\Omega$, let  $S$ be a non-abelian simple group and let $G$ be a group with $S\mathrm{wr}_\Omega K \leq G \leq (\aut S) \mathrm{wr}_\Omega K.$ Then
	$$\sum_{\substack{p\in \pi^*(S)}}d_p(G)>t_\Omega(K).$$
\end{thm}

\begin{proof}For every $p\in \pi^*(S)$, we have $d_p(G)=d_p(S\mathrm{wr}_\Omega K)$ and hence,
 without loss of generality, we may assume $G=S \mathrm{wr}_\Omega K.$ For simplicity, we write $$f(S,\Omega,K):=\sum_{\substack{p\in \pi^*(S)}}d_p(G).$$
	
	We argue by induction on $t:=t_\Omega(K)$. When $t=1$, from Theorem~\ref{almeno2} we deduce $$f(S,\Omega,K)\ge \pi^*(S)\ge 2>1=t.$$
	Suppose then $t>1$. Let  $\omega\in \Omega$ and  let  $U_1,\ldots,U_t$ be $t$ independent subgroups of $K$  with $$\bigcap_{i=1}^tU_i=K_\omega.$$
	
	For each $i\in \{1,\ldots,t\}$, we define
	\begin{itemize}
		\item[$\bar{U}_i$]  to be the intersection $\underset{j\in\{1,\dots,t\}\setminus \{i\}}\bigcap U_j;$ (as $K_\omega\le\bar{U}_i$, the orbit $\omega^{\bar{U}_i}:=\{\omega^x\mid x\in \bar{U}_i\}$ is a block of imprimitivity for the action of $K$ on $\Omega$.)
		\item[$\Omega_i$] to be the system of imprimitivity determined by the block of imprimitivity $\omega^{\bar{U}_i}$;
		\item[$\hat{K}_i$] to be the permutation group induced by $K$ on $\Omega_i$; (we also denote by  $\sigma_i:K\to \hat{K}_i$ the natural projection, so  $\hat{K}_i=\sigma_i(K)$.)
		\item[$G_i$] to be the wreath product $G_i:=S\mathrm{wr}_{\Omega_i} \hat{K}_i$.
	\end{itemize}
	
	Let $i\in \{1,\ldots,t\}$. Since the point stabilizer $\sigma_i(\bar{U}_i)$ of $\omega^{\bar{U}_i}\in \Omega_i$ in $\hat{K}_i$ is defined as the intersection of the $t-1$ independent subgroups $\{\sigma_i(U_j)\mid j \in \{1,\ldots,t\}\setminus\{i\}\}$, we have $t_{\Omega_i}(\hat{K}_i)\ge t-1$.
	Moreover, from our inductive argument,  we have 
	\begin{equation}\label{f2}\sum_{p\in\pi^*(S)}d_p(G_i)=f(S,\Omega_i,\hat{K}_i)>t_{\Omega_i}(\hat{K}_i)\ge t-1.
	\end{equation}

	For each prime $p\in \pi^*(S)$, let $\Pi_p$ be a Sylow $p$-subgroup of $S$ and let $P$ be a Sylow $p$-subgroup of $K$. In particular, $\hat{P}_i:=\sigma_i(P)$ is a Sylow $p$-subgroup of $\hat{K}_i$. From Lemma~\ref{l:1}, for every $i\in ~\{1,\ldots,t\}$, we have 
	\begin{equation}\label{f1}
	f(S,m,\hat{K}_i)=\sum_{p\in \pi^*(S)}(d(\hat{P}_i)+n_{\Omega_i}(\hat{P}_i)d(\Pi_p)),
	\end{equation}
	where $n_{\Omega_i}(\hat{P}_i)=n_{\Omega_i}(P)$ denotes the number of orbits  of $P$ on $\Omega_i$.  Observe that $d(P)\ge d(\hat{P}_i)$. 
	
	In particular, using~\eqref{f2} and~\eqref{f1}, we deduce
	$$f(S,m,K)>t,$$ 
	unless, for each $i\in \{1,\ldots,t\}$ and for each  $p\in \pi^*(S)$,
	\begin{description}
		\item[(a)] $d(P)=d(\hat{P}_i)$,
		\item[(b)] $n_\Omega(P)=n_{\Omega_i}(P)$.
	\end{description}
	In particular, for the rest of the proof we may assume that~{\bf (a)} and~{\bf (b)} hold.
	
Since $|\pi^*(S)|\ge 2$, we may choose $p\in \pi^*(S)$ and $i\in\{1,\ldots,\ell\}$ such that $|\bar{U}_i:K_\omega|$ is not a power of $p$. Let $\hat{\delta}_1,\ldots,\hat{\delta}_s$ be a set of representatives  of the orbits of $P$ on $\Omega_i$, where $s:=n_{\Omega_i}(P)$. In other words, this means that $$\Omega_i=\bigcup_{j=1}^s\{\hat{\delta}_j^x\mid x \in P\}$$ and that this union is disjoint. For each $j\in\{1,\ldots,s\}$,  let $\delta_j\in \hat{\delta}_j$.  As $\hat{\delta}_j\subseteq\Omega$ is a block of imprimitivity for the action of $K$ on $\Omega$, the union
	\begin{equation}\label{f22}\bigcup_{j=1}^s\{\delta_j^x\mid x\in P\}\subseteq \Omega
	\end{equation} 
	is made by pairwise disjoint $P$-orbits and hence $n_\Omega(P)\ge s=n_{\Omega_i}(P)$. Moreover, $n_\Omega(P)=n_{\Omega_i}(P)$ if and only if the equality in~\eqref{f22} is attained, which in turn happens, if and only if,  for each $j\in \{1,\ldots,s\}$, the points in $\hat{\delta}_j\subseteq\Omega$ are in the same $P$-orbit.
	
	Since we are assuming that $n_\Omega(P)=n_{\Omega_i}(P)$, the previous paragraph shows that the stabilizer $P_{\hat{\delta}_j}$ of the block $\hat{\delta}_j$ is transitive on the points in $\hat{\delta}_j$. Since $P$ is a $p$-group, we deduce 
	$|\hat{\delta}_j|=|\bar{U}_i:K_\omega|$  is a power of $p$, contradicting our choice of $i$ and $p$. 
\end{proof}

\section{Proofs of Theorem \ref{main} and Corollary \ref{coro}}\label{proofs}

If $N$ is a normal subgroup of a finite group $G$, we  denote by $m(G,N)$ the difference $m(G)-m(G/N).$ We recall in the first part of this section some results proved in \cite{min1,min2}, estimating the value of $m(G,N)$ when $N$ is a minimal normal subgroup of $G.$

\begin{lemma}
If $N$ is an abelian minimal normal subgroup of $G,$ then $m(G,N)$ is either 0 or 1  depending on whether $N\leq \frat (G)$ or not.
\end{lemma} 
\begin{proof} If follows from \cite[Lemma 11 and Lemma 12]{min1}.
\end{proof}

\begin{lemma}\label{senzal}Assume that $N$ is a non-abelian minimal normal subgroup of a finite group $G$. There exist a
non-abelian simple group $S$ and a positive integer $r$ such that $N = 
S_1\times \dots \times S_r,$ with $S\cong S_i$ for each $1 \leq i \leq r$. 
Let $K$ be the transitive subgroup of $\perm(r)$ induced by the conjugacy action of 
$G$ on the set $\{S_1, \dots , S_r\}$ of the simple components of $N$. As in the previous section, let $t(K):=t_{\{1,\cdots,r\}}(K)$ be the largest positive integer $t$ such that the stabilizer in $K$ of a point in $\{1,\cdots,r\} $ can be obtained as an intersection of $t$ independent subgroups. 
Moreover let
$X$ be the subgroup  of $\aut S_1 $ induced
by the conjugation action of $N_G(S_1)$ on the first factor $S_1$ . Then
$$
m(G,N)\leq m(X,\soc X)+t(K).
$$
\end{lemma}
\begin{proof} If follows from \cite[Lemma 13]{min1} and \cite[Theorem 1]{min2}.
\end{proof}

\begin{lemma}\label{lemmalemma}
	Let $N$ be a minimal normal subgroup of a finite group $G.$ If $N\not\leq \frat(G),$ then
	$\delta(G)\geq \delta(G/N)+|\pi(N)|.$
\end{lemma}
\begin{proof}
	It suffice to prove that $d_p(G)>d_p(G/N)$ whenever $p\in \pi(N).$ This is clear when $N$ is abelian.  Assume that $N$ is non-abelian. Let $p\in \pi(N)$  and let $P$ be a Sylow
	$p$-subgroup of $G$. If $P\cap N\not\leq \frat (P)$, then Tate's Theorem \cite[p. 431]{hup} shows that $N$ has a normal $p$-complement. However, this is impossible because $N$ is a direct product of non-abelian simple groups. Thus $P\cap N\not\leq \frat (P),$ and consequently $d_p(G/N)+1\leq d_p(G).$
\end{proof}

\begin{proof}[Proof of Theorem~$\ref{main}$]
Clearly the statement is true if $G$ is simple. Thus we suppose that $S$ is not a simple group and we proceed by induction on the order of $G.$ We may assume $\frat(G)=1.$ Let $N$ be a minimal normal subgroup of $G.$ If $N$ is abelian, using Lemma~\ref{lemmalemma} and the inductive hypotheses, we have
	$$m(G)=m(G/N)+1\leq \sigma(\delta(G/N))^\eta+1\leq \sigma\cdot (\delta(G)-1)^\eta+1\leq \sigma \cdot \delta(G)^{\eta}.$$
(In the last inequality, we used the fact that $\sigma\ge 1$ and $\eta\ge 2$.)	Assume that $N$ is non-abelian. Let $K, X$ and $S$ be as in the statement of Lemma \ref{senzal}. By Theorem~\ref{pablo}, we have $$t(K)<\sum_{p\in\pi^*(S)}d_p(G)\leq \delta(G).$$ Combining this with Lemma~\ref{lemmalemma}, we conclude that
	$$\begin{aligned}m(G)&\leq m(G/N)+m(X,S)+t(K)\leq \sigma\cdot \delta(G/N)^\eta+\sigma\cdot |\pi(S)|^\eta+\delta(G)\\
	& \leq \sigma\cdot \delta(G/N)^\eta+\sigma\cdot |\pi(S)|^\eta+\sigma\cdot \delta(G) \leq\sigma(\delta(G/N)^\eta+|\pi(N)|^\eta+\delta(G))\\
	& \leq	\sigma((\delta(G/N)^\eta+(\delta(G)-\delta(G/N))^\eta+\delta(G/N)+(\delta(G)-\delta(G/N))
	)\\
	&\leq \sigma \cdot \delta(G)^\eta.
	\end{aligned}
	$$
The last inequality follows from the fact  that $x^\eta+y^\eta+x+y \leq(x+y)^\eta$, for every positive integers $x$ and $y$  and for every $\eta\geq 2.$ 
\end{proof}

In order to prove Corollary \ref{coro}, we first need the following lemma.

\begin{lemma}\label{stup}For every positive real number $\eta > 1,$ there exists a constant $c_\eta$ such that $n \leq c_\eta\pi(n)^\eta,$ where $\pi(n)$ is the number of prime numbers less than or equal to $n$.
\end{lemma}
\begin{proof}
	By \cite[Theorem 29]{ros}, if $n\geq 55,$ then
	$\pi(n)>\frac{n}{\log n+2},$  so if suffices to notice that
	$\lim_{n\to\infty}\frac{n^{\eta-1}}{(\log n+2)^\eta}=\infty.$           
	\end{proof}
\begin{lemma}\label{simple}
	There exists a constant $\rho$ such that, if $X$ is an almost simple group and $S=\soc(X)$ is not a simple group of Lie type, then $m(X,S)\leq \rho \cdot |\pi(S)|^2.$
\end{lemma}
\begin{proof}
	First assume that $S=\alt(n)$. By \cite[Theorem 1]{w}, $m(X,S)\leq n-1$. By Lemma~\ref{stup}, there exists  a constant $c_2$ such that $m(X,S)\leq c_2 {\pi(n)}^2= c_2  \cdot {|\pi(S)|}^2. $
	Clearly there exists a constant $c$ such that $m(X,S)\leq c \cdot|\pi(S)|^2$, for every  sporadic simple group $S$. Taking $\rho=\max\{c,c_2\},$ the result follows.
\end{proof}

\begin{proof}[Proof of Corollary~$\ref{coro}$]
	It follows from Theorem~\ref{main} and Lemma~\ref{simple}.
\end{proof}

\section{Estimating $\delta(\perm(n))$}\label{conti}
In this section, we aim to bound, from above and from below, $\delta(\perm(n))$ as a function of $n$.  By \cite[Theorem 1]{w}, $m(\perm(n))=n-1$ while, by Kalu\v{z}nin's Theorem, if
$$a_{\ell(p,n)} p^{\ell(p,n)}+a_{\ell(p,n)-1}p^{\ell(p,n)-1}+\cdots +a_1p+a_0$$ is the $p$-adic expansion of $n$, then $$d_p(\perm(n))=a_{\ell(p,n)} \ell(p,n)+a_{\ell(p,n)-1}(\ell(p,n)-1)+\cdots+a_1.$$
For not making  the notation too cumbersome, we set  
$$d_p(n):=d_p(\perm(n))=a_{\ell(p,n)} \ell(p,n)+a_{\ell(p,n)-1}(\ell(p,n)-1)+\cdots+a_1$$
and
$$d(n):=\sum_{p\textrm{ prime}}d_p(n)=\delta(\perm(n)).$$

As in the previous sections we denote by $\pi:\mathbb{R}\to\mathbb{N}$ the prime counting function, that is, $\pi(x)$ is the number of prime numbers less than or equal to $x$. As $d_p(n)\ge \ell(p,n)\ne 0$ for every prime $p\le n$, we have
\begin{align*}
\pi(n)\le d(n).
\end{align*}
From the Prime Number Theorem, $\pi(n)$ is asymptotic to $n/\log n$ (that is, the ratio $\pi(n)/(n/\log n)$ tends to $1$ as $n$ tends to infinity) and hence $n/\log n\in O(d(n))$. In this section, we actually prove that $d(n)$ is asymptotic to a linear function.
\begin{thm} \label{stop}For every $n\ge 2$, we have
$$n\log 2-\frac{12n}{\log n}\le d(n)\le n\log 2+\frac{19n}{2\log n}+\frac{137n}{2\log^2n}+\frac{4\sqrt{n}}{\log n}+\frac{3\sqrt{n}}{2}\log n\le n\log 2+\frac{112n}{\log n}.$$
	In particular,  $d(n)=n\log 2+O(n/\log n).$
\end{thm}
\begin{proof}
We start by collecting some basic inequalities that we use throughout this proof.
From Theorem~1 and Theorem~2 in \cite{RS}, we have
\begin{align}
\label{bo1}&\pi(x)\leq \frac{x}{\log x}\left(1+\frac{3}{2\log x}\right),\qquad\; \forall x>1,\\
&\pi(x)\geq\frac{x}{\log x-1/2},\qquad\qquad\qquad \forall x\geq 67.\label{bo3}
\end{align}
Given a prime number $p$ with $p\le n$, $\ell(p,n)\le\lfloor\log_p n\rfloor$ and hence
\begin{align}\label{P1}\nonumber
d_p(n)&\le(p-1)(\ell(p,n)+(\ell(p,n)-1)+\cdots+2+1)
\\&=(p-1)\frac{\ell(p,n)(\ell(p,n)+1)}{2}\le (p-1)\frac{\log_p n(\log_p n+1)}{2}.
\end{align}

We define the two auxiliary functions $$d'(n):=\sum_{p\le \sqrt{n}}d_p(n);\quad d''(n):=\sum_{\sqrt{n}<p\le n}d_p(n).$$ 
We aim to obtain explicit bounds on $d'(n)$ and $d''(n)$ as functions of $n$. We start with $d'(n)$.

From \eqref{P1}, we get
\begin{align}\label{eq:again11}
d'(n)&\le\frac{\log^2n}{2}\sum_{p\le\sqrt{n}}\frac{p-1}{\log^2p}+\frac{\log n}{2}\sum_{p\le \sqrt{n}}\frac{p-1}{\log p}.
\end{align}
For every $k\in \mathbb{N}$ with $k\ge 1$, we denote by $p_k$ the $k^{\mathrm{th}}$ prime number. Using~\cite[Corollary, page~69]{RS}, we have
\begin{align*}
k\log k<p_k&<k(\log k+\log \log k),
\end{align*}
where the first inequality is valid for every $k\geq 1$ and the second inequality is valid for every $k\ge 6$.

This shows that, for every $k\geq 6$,
\begin{align}\label{eq:P3}
\frac{p_k-1}{\log p_k}\le \frac{k(\log k+\log \log k)}{\log(k\log k)}=k.
\end{align}
An explicit computation yields that~\eqref{eq:P3} is also valid when $k\in \{2,3,4,5\}$.

 Therefore, from~\eqref{bo1},~\eqref{eq:P3} and a computation (we are using $n\ge 11$ in the last inequality), for every $n\geq 11$,  we have
\begin{equation*}
\begin{split}
\sum_{p\le \sqrt{n}}\frac{p-1}{\log p}&=\frac{1}{\log 2}+\sum_{2<p\le \sqrt{n}}\frac{p-1}{\log p}=\frac{1}{\log 2}+\sum_{k=2}^{\pi(\sqrt{n})}\frac{p_k-1}{\log p_k}\le\frac{1}{\log 2}+\sum_{k=2}^{\pi(\sqrt{n})}k\\
&=\frac{1}{\log 2}+\frac{\pi(\sqrt{n})(\pi(\sqrt{n})+1)}{2}-1=\frac{1}{\log 2}-1+\frac{\pi(\sqrt{n})^2}{2}+\frac{\pi(\sqrt{n})}{2}\\
&\le -1+\frac{1}{\log 2}+\frac{1}{2}\left(\frac{\sqrt{n}}{\log \sqrt{n}}\left(1+\frac{3}{2\log \sqrt{n}}\right)\right)^2+\frac{\sqrt{n}}{2\log \sqrt{n}}\left(1+\frac{3}{2\log \sqrt{n}}\right)\\
&=-1+\frac{1}{\log 2}+\frac{1}{2}\left(\frac{4n}{\log^2n}+\frac{36n}{\log^4n}+\frac{24n}{\log^3n}\right)+\frac{\sqrt{n}}{\log n}+\frac{3\sqrt{n}}{\log^2n}\\
	&\leq\frac{2n}{\log^2 n}+ \frac{24n}{\log^3n}.\\
\end{split}
\end{equation*}
Actually, with a direct inspection, we see that this inequality holds true for every natural number $n$ with $2\le n \le 10$. Therefore
\begin{align}\label{eq:again1}
\sum_{p\le \sqrt{n}}\frac{p-1}{\log p}\leq\frac{2n}{\log^2n}+\frac{24n}{\log^3n},
\end{align}
for every $n>1$.

Arguing in a similar manner, for every $k\ge 6$, we obtain
\begin{align}\label{eq:4}
\frac{p_k-1}{\log^2 p_k}\le \frac{k(\log k+\log \log k)}{\log^2(k\log k)}=\frac{k}{\log k+\log \log k}.
\end{align}
An explicit computation yields that~\eqref{eq:4} is also valid when $k\in \{2,3,4,5\}$. Therefore, using~\eqref{eq:4}, we have
\begin{align*}
\sum_{p\le \sqrt{n}}\frac{p-1}{\log^2 (p)}&\le\frac{1}{\log^2(2)}+\sum_{k=2}^{\pi(\sqrt{n})}\frac{k}{\log k+\log \log k}.
\end{align*}
For every $t\in \mathbb{N}$ with $t\ge 2$, write $f(t):=\sum_{k=2}^tk/(\log k+\log \log k)$. When $k>2$, we have $k/(\log k+\log \log k)\le k$. Moreover, when $k\ge \sqrt{t}$, we have
\begin{align*}
\frac{k}{\log k+\log \log k}&\le \frac{k}{\log \sqrt{t}+\log\log(\sqrt{t})}=\frac{k}{\log t/2+\log(\log t)-\log 2}\\
&\le \frac{2k}{\log t},
\end{align*}
where the last inequality holds for $t\ge 8$. Therefore, for every $t\ge 8$, we have
\begin{align*}
f(t)&=\frac{2}{\log 2+\log\log 2}+\sum_{2<k\le \sqrt{t}}\frac{k}{\log k+\log\log k}+\sum_{\sqrt{t}<k\le t}\frac{k}{\log k+\log\log k}\\
&\le \frac{2}{\log 2+\log \log 2}+\sum_{2<k\le \sqrt{t}}k+\sum_{\sqrt{t}<k\le t}\frac{2k}{\log t}\\
&\le \frac{2}{\log 2+\log \log 2}+\frac{\sqrt{t}(\sqrt{t}+1)}{2}-3+\frac{t(t+1)}{\log t}\le \frac{t^2}{\log t}+t,
\end{align*}
where the last inequality follows with some elementary computations. A direct computation with $2\le t<8$ shows that the same upper bound for $f(t)$ holds.
Therefore, applying this upper bound with $t:=\pi(\sqrt{n})$, we get
\begin{align}\label{eq:new}
\sum_{p\le \sqrt{n}}\frac{p-1}{\log^2 (p)}&\le\frac{1}{\log^2 2}+f(\pi(\sqrt{n}))\le\frac{1}{\log^2 2}+\frac{\pi(\sqrt{n})^2}{\log \pi(\sqrt{n})}+\pi(\sqrt{n}).
\end{align}
Now, for every $n\ge 67^2$, using \eqref{bo1} and \eqref{bo3}, we see that the right hand side of (\ref{eq:new}) is bounded above by
\begin{equation}
 \begin{split}\label{eq:newnew}
&\frac{1}{\log^2 2}+\frac{\left(\frac{\sqrt{n}}{\log\sqrt{n}}\left(1+\frac{3}{2\log \sqrt{n}}\right)\right)^2}{\log \left(\frac{\sqrt{n}}{\log\sqrt{n}-1/2}\right)}+\frac{\sqrt{n}}{\log\sqrt{n}}\left(
1+\frac{3}{2\log \sqrt{n}}\right).\\
\end{split}
\end{equation}

The second summand of~\eqref{eq:newnew} is at most
\begin{align*}
&\frac{\frac{4n}{\log^2 n}\left(1+\frac{3}{\log n}\right)^2}{\log(\sqrt{n}/\log \sqrt{n})}.
\end{align*}
Now,  we have $\log(\sqrt{n}/\log \sqrt{n})>\log n/4$. Thus the second summand of~\eqref{eq:newnew} is at most
\begin{align*}
\frac{16 n}{\log^3n}+\frac{96n}{\log^4 n}+\frac{144n}{\log^5n}\le \frac{16n}{\log^3 n}+\frac{114n}{\log^4n},
\end{align*}
where the last inequality follows with a computation using the fact that $n\ge 67^2$.
For the first and third summand of~\eqref{eq:newnew}, we have
\begin{align*}
\frac{1}{\log^2(2)}+\frac{2\sqrt{n}}{\log(n)}+\frac{6\sqrt{n}}{\log^2(n)}<\frac{3\sqrt{n}}{\log(n)},
\end{align*}
where this inequality follows again with some elementary computations using the fact that $n\ge 67^2$. Summing up, for every $n\ge 67^2$, we have
\begin{align}\label{eq:again2}
\sum_{p\le\sqrt{n}}\frac{p-1}{\log^2 p}\le\frac{16 n}{\log^3n}+\frac{114n}{\log^4n}+\frac{3\sqrt{n}}{\log n}.
\end{align}
A direct inspection  shows that this bound is also valid for the natural numbers $n$ with $n\le 67^2$.

Summing up, from~\eqref{eq:again11},~\eqref{eq:again1} and~\eqref{eq:again2}, we get
\begin{equation}
\begin{split}\label{eq:def1}
d'(n)&\leq \frac{8 n}{\log n}+\frac{57n}{\log^2n}+\frac{3}{2}\sqrt{n}\log(n)+\frac{n}{\log n}+\frac{12n}{\log^2n}\\
&=\frac{9 n}{\log n}+\frac{69n}{\log^2n}+\frac{3}{2}\sqrt{n}\log n.
\end{split}
\end{equation}

We now start working on the function $d''(n)=\sum_{\sqrt{n}<p\le n}d_p(n)$. Here we are interested in a lower bound and in an upper bound for $d''(n)$.  First we obtain an upper bound for $d''(n)$.
As $p>\sqrt{n}$, the $p$-adic expansion of $n$ is simply $n:=a_1(p,n)p+a_0$ and hence $d_p(n)=a_1(p,n)$. Now we refine further $d''(n)$. For every $i\in  \{1,\ldots,\lfloor\sqrt{n}\rfloor-1\}$, we let
	$$g_i(n):=\sum_{n/(i+1)< p\le n/i}a_1(p,n)$$
	and we let
	$$g_{\lfloor\sqrt{n}\rfloor}(n):=\sum_{\sqrt{n}< p\le n/\lfloor\sqrt{n}\rfloor}a_1(p,n).$$
	When $i\in \{1,\ldots,\lfloor\sqrt{n}\rfloor\}$, we have $a_1(p,n)=i$ and hence $g_i(n)$ equals $i$ times the number of prime numbers in the interval $(n/(i+1),n/i]$. Therefore, when $i\in \{1,\ldots,\lfloor\sqrt{n}\rfloor-1\}$,
	\begin{align*}
	g_i(n)&=i(\pi(n/i)-\pi(n/(i+1)))
	\end{align*}
	and 
	\begin{align*}
	g_{\lfloor\sqrt{n}\rfloor}(n)=\lfloor\sqrt{n}\rfloor(\pi(n/\lfloor\sqrt{n}\rfloor)-\pi(\sqrt{n})).
	\end{align*}
Since every prime $p$, with $\sqrt{n}<p\le n$, lies in one of the intervals $(n/(i+1),n/i]$, for some $i\in \{1,\ldots,\lfloor \sqrt{n}\rfloor-1\}$, or in the interval $(\sqrt{n},n/\lfloor \sqrt{n}\rfloor]$, we have  
	\begin{align}\label{eq:dm2}
	d''(n)&=\sum_{i=1}^{\lfloor\sqrt{n}\rfloor}g_i(n)=\sum_{i=1}^{\lfloor \sqrt{n}\rfloor-1}i(\pi(n/i)-\pi(n/(i+1)))+\lfloor\sqrt{n}\rfloor(\pi(n/\lfloor\sqrt{n}\rfloor)-\pi(\sqrt{n}))\\\nonumber
	&=\sum_{i=1}^{\lfloor \sqrt{n}\rfloor}\pi(n/i)-\lfloor \sqrt{n}\rfloor\pi(\sqrt{n}).
	\end{align}
	 Using~\eqref{bo1}, we have
	\begin{align}\label{eq:dm3}
	\sum_{i=1}^{\lfloor \sqrt{n}\rfloor}\pi(n/i)&\le 
	\sum_{i=1}^{\lfloor \sqrt{n}\rfloor }\frac{n/i}{\log (n/i)}\left(1+\frac{3}{2\log (n/i)}\right)\\\nonumber
	&= 
	\sum_{i=1}^{\lfloor \sqrt{n}\rfloor }\frac{n/i}{\log (n/i)}+\frac{3}{2}\sum_{i=1}^{\lfloor \sqrt{n}\rfloor }\frac{n/i}{\log^2 (n/i)}.
	\end{align}
	The function $x\mapsto (n/x)/\log(n/x)$ is decreasing in the interval $(0,\lfloor\sqrt{n}\rfloor]$ and hence we obtain for the first summand the estimate
	\begin{align}\label{eq:dm4}
	\sum_{i=1}^{\lfloor \sqrt{n}\rfloor }\frac{n/i}{\log (n/i)}&=\frac{n}{\log n}+\sum_{i=2}^{\lfloor \sqrt{n}\rfloor }\frac{n/i}{\log (n/i)}\le \frac{n}{\log n}+\int_1^{\lfloor\sqrt{n}\rfloor}\frac{n/x}{\log(n/x)}dx\\\nonumber
	&=\frac{n}{\log n}+\left[-n\log(\log(n/x))\right]_1^{\lfloor\sqrt{n}\rfloor}\\\nonumber
	&=\frac{n}{\log{n}}-n\log\log(n/\lfloor \sqrt{n}\rfloor)+n\log(\log n).
	\end{align}
	For the second summand observe that the function $x\mapsto (n/x)/\log^2(n/x)$ is decreasing in the interval $(0,\lfloor\sqrt{n}\rfloor]$ and hence we obtain the estimate
	\begin{align}\label{eq:dm5}
	\frac{3}{2}\sum_{i=1}^{\lfloor \sqrt{n}\rfloor }\frac{n/i}{\log^2 (n/i)}
	&=\frac{3n}{2\log^2 n}+\frac{3}{2}\sum_{i=2}^{\lfloor \sqrt{n}\rfloor }\frac{n/i}{\log^2 (n/i)}\\\nonumber
	&\le\frac{3n}{2\log^2 n}+\frac{3}{2}\int_1^{\lfloor\sqrt{n}\rfloor}\frac{n/x}{\log^2(n/x)}dx\\\nonumber&=\frac{3n}{2\log^2(n)}+\frac{3}{2}\left[\frac{n}{\log(n/x) }\right]_1^{\lfloor\sqrt{n}\rfloor}\\
	&=\frac{3n}{2\log^2(n)}+\frac{3n}{2\log(n/\lfloor\sqrt{n}\rfloor)}-\frac{3n}{2\log n}.\nonumber
	\end{align}
	Further, for $n\geq 67^2$, we get
	\begin{align}\label{eq:new6}
\lfloor \sqrt{n}\rfloor\pi(\sqrt{n})&\ge (\sqrt{n}-1)\frac{\sqrt n}{\log \sqrt n-1/2}=\frac{2n}{\log n-1}-\frac{2\sqrt{n}}{\log n-1}\\\nonumber
&=\frac{2n}{\log n}+2n\left(\frac{1}{\log n-1}-\frac{1}{\log n}\right)-\frac{2\sqrt{n}}{\log n-1}\\\nonumber
&\ge \frac{2n}{\log n}+\frac{2n}{\log n(\log n-1)}-\frac{2\sqrt{n}}{\log n/2}\\\nonumber
&\ge  \frac{2n}{\log n}+\frac{2n}{\log^2n}-\frac{4\sqrt{n}}{\log n}.
	\end{align}
	Thus, from~\eqref{eq:dm2},~\eqref{eq:dm3},~\eqref{eq:dm4},~\eqref{eq:dm5} and~\eqref{eq:new6}, for every $n\ge 67^2$, we have that 
	\begin{align*}
	&d''(n)\leq n\log(\log n)-n\log(\log(n/\lfloor\sqrt{n}\rfloor))-\frac{n}{2\log n}+\frac{3n}{2\log^2 n}+\frac{3n}{2\log(n/\lfloor\sqrt{n}\rfloor)}\\
	&\qquad\quad\,\, -  \frac{2n}{\log n}-\frac{2n}{\log^2n}+\frac{4\sqrt{n}}{\log n}.
	\end{align*}
	First of all, as $n/\lfloor\sqrt{n}\rfloor\ge \sqrt{n}$, we get $\log(n/\lfloor\sqrt{n}\rfloor)\ge \log\sqrt{n}=\log(n)/2$ and hence
	$$-\frac{n}{2\log n}+\frac{3n}{2\log(n/\lfloor\sqrt{n}\rfloor)}-\frac{2n}{\log n}\le \left(-\frac{1}{2}+3-2\right)\frac{n}{\log n}=\frac{n}{2\log n}.$$
	Moreover,
	\begin{align*}n\log(\log n)-n\log(\log(n/\lfloor\sqrt{n}\rfloor))&\le n\log\log n-n\log\log(\sqrt{n})\\
	&=n\log\left(\frac{\log n}{\log \sqrt{n}}\right)=n\log 2.
	\end{align*}
	Summing up, for every $n\ge 67^2$,
	\begin{align}\label{eq:final1}
	d''(n)\le n\log 2+\frac{n}{2\log n}-\frac{n}{2\log^2 n}+\frac{4\sqrt{n}}{\log n}.
	\end{align}
An explicit computation with the positive integers $n$ with $2\le n<67^2$ shows that the same upper bound remains true when $n\leq 67^2$.	

Using the upper bounds~\eqref{eq:def1} and~\eqref{eq:final1}, for every $n\ge 2$, we deduce
$$d(n)=d'(n)+d''(n)\le n\log 2+\frac{19n}{2\log n}+\frac{137n}{2\log^2n}+\frac{4\sqrt{n}}{\log n}+\frac{3\sqrt{n}}{2}\log n\le n\log 2+\frac{112n}{\log n},$$
where the last inequality follows with some computation.

Now,	we use the argument above to obtain also a  lower bound for $d''(n)$ and hence for $d''(n)$. Using~\eqref{bo3} and~\eqref{eq:dm2}, we have
	\begin{align*}
	d''(n)&=\sum_{i=1}^{\lfloor \sqrt{n}\rfloor}\pi(n/i)-\lfloor\sqrt{n}\rfloor\pi(\sqrt{n})\ge 
	\sum_{i=1}^{\lfloor \sqrt{n}\rfloor }\frac{n/i}{\log (n/i)-1/2}-\lfloor\sqrt{n}\rfloor\pi(\sqrt{n})\nonumber\\
	&\ge \sum_{i=1}^{\lfloor \sqrt{n}\rfloor }\frac{n/i}{\log (n/i)}-\sqrt{n}\pi(\sqrt{n}).\nonumber
	\end{align*}
	The function $x\mapsto (n/x)/\log(n/x)$ is decreasing in the interval $(0,\lfloor\sqrt{n}\rfloor]$ and hence we obtain the estimate
	\begin{align*}
	\sum_{i=1}^{\lfloor \sqrt{n}\rfloor }\frac{n/i}{\log (n/i)}&\ge\int_1^{\lfloor\sqrt{n}\rfloor}\frac{n/x}{\log(n/x)}dx=\left[-n\log(\log(n/x))\right]_1^{\lfloor\sqrt{n}\rfloor}\\\nonumber
	&=-n\log\log(n/\lfloor \sqrt{n}\rfloor)+n\log(\log n)=n\log\left(\frac{\log n}{\log(n/\lfloor \sqrt{n}\rfloor)}\right)\\
&=n\log\left(\frac{\log n}{\log n-\log( \lfloor\sqrt{n} \rfloor ) }\right)=n\log\left(\frac{\log n}{\log n-\log \sqrt{n}-\log( \lfloor\sqrt{n} \rfloor/\sqrt{n} ) }\right)\\
&=n\log\left(\frac{\log n}{(\log n)/2-\log( \lfloor\sqrt{n} \rfloor/\sqrt{n} ) }\right)\ge n\log 2,
	\end{align*}
where in the last inequality we used the fact that $ \lfloor\sqrt{n} \rfloor/\sqrt{n} \le 1$ and hence $\log ( \lfloor\sqrt{n} \rfloor/\sqrt{n})\le 0$. 
Furthermore, from \eqref{bo1}, we have
\begin{align*}
\sqrt{n}\pi(\sqrt{n})\le \frac{n}{\log \sqrt{n}}\left(1+\frac{3}{2\log \sqrt{n}}\right)=\frac{2n}{\log n}\left(1+\frac{3}{\log n}\right)\le \frac{12n}{\log n},
\end{align*}
where the last inequality follows from an easy computation. Summing up,
$$d(n)=d'(n)+d''(n)\ge d''(n)\ge n\log 2-\frac{12n}{\log n}.\qedhere$$
\end{proof}

\end{document}